\documentclass{amsart}
\usepackage{amssymb}
\usepackage{amsmath}
\usepackage{amsthm}

\theoremstyle{plain}
\newtheorem{theorem}{Theorem}[section]
\newtheorem{case}{Case}
\newtheorem{corollary}[theorem]{Corollary}
\newtheorem{lemma}[theorem]{Lemma}
\newtheorem{remark}[theorem]{Remark}

\numberwithin{equation}{section}

\begin{document}
\title[Numbers with countable expansions] {Numbers with countable expansions in base of generalized golden ratios}
\author{YUEHUA GE \& BO TAN$^{\dag}$}
\address{Huazhong University of Science and Technology, Wuhan, 430074, P.R. China}
\email{geyuehua1001@126.com, tanbo@mail.hust.edu.cn }
\thanks{$^{\dag}$Corresponding author.}
\keywords{beta-expansions, countable expansions, generalized golden ratio}
\begin{abstract}
Sidorov and Vershik showed that in base $G=\frac{\sqrt{5}+1}{2}$ and with the digits $0,1$ the numbers $x=nG ~(\text {mod} 1)$ have $\aleph_{0}$ expansions for any $n\in\mathbb{Z}$, while the other elements of $(0, \frac{1}{G-1})$ have $2^{\aleph_{0}}$ expansions. In this paper, we generalize this result to the generalized golden ratio base $\beta=\mathcal{G}(m)$. With the digit-set  $\{0,1,\cdots, m\}$, if
  $m=2k+1$, $\mathcal{G}(m)=\frac{k+1+\sqrt{k^{2}+6k+5}}{2}$, the numbers $x=\frac{p\beta+q}{(k+1)^{n}}\in(0, \frac{m}{\beta-1})$ (where $n, p, q\in\mathbb{Z}$) have $\aleph_{0}$ expansions, while the other elements of $(0, \frac{m}{\beta-1})$ have $2^{\aleph_{0}}$ expansions;
  if $m=2k$, $\mathcal{G}(m)=k+1$, the numbers with countably many expansions are $\frac{p}{(k+1)^{n}}\in(0, 2) ~(n, p\in\mathbb{N}\cup\{0\})$.
 This solves an  open  question by Baker.
\end{abstract}
\maketitle

\section {introduction}

The  $\beta-$expansion extends the representation of real numbers from the integer base (e.g. the familiar decimal or binary expansions) to the non-integer base. It was introduced by  R\'{e}nyi \cite{R}, and was developed by Parry \cite{P}. Since then this kind of  expansion was extensively studied from various viewpoints, and these studies have connections with many  fields such as topology, symbolic dynamical system and combinatorics  \cite{B, DR, IT, KL4,  NS, S}.

Given a base and a digit-set,  we consider all the possible expansions of the numbers (rather than the greedy expansion only).
In this setting, some authors have studied that the set of real numbers which admit a unique expansion from various aspects such as the topological structure, the metric property, or the fractal dimension, see
  e.g. \cite{BK, EHJ, K, KL, KL2, LTW, KL3, VMK, VM}.

  \smallskip

  Let $\Omega$ be a finite set, called an alphabet or a digit-set. For any  $n\in\mathbb{N}$,  we put $$\Omega^{n}=\{ x_{1}x_{2}\ldots x_{n}: x_{i}\in\Omega \text{ for }i=1, 2, \ldots, n\},$$
and
$$\Omega^{\ast}=\bigcup_{n\geq0}\Omega^{n}, $$
  here, we put, by convention, that $\Omega^0=\{\emptyset\}$ with $\emptyset$ the empty word. We also set
  $$\Omega^{\infty}=\{ x_{1}x_{2}x_{3}\ldots: x_{i}\in\Omega  \text{ for } i\geq1\}. $$

Let $m\in\mathbb{N}, \beta\in(1, m+1]$ and $I_{\beta, m}=[0, \frac{m}{\beta-1}]$. It is easy to see that each $x\in I_{\beta, m}$ has an expansion of the form
\begin{equation}\label{def}
   x=\sum_{i=1}^{\infty}\frac{\varepsilon_{i}}{\beta^{i}}
\end{equation}
 for some $(\varepsilon_{i})_{i=1}^{\infty}\in\{0, 1, \ldots, m\}^{\infty}$, whence we say that  the sequence $0. \varepsilon_{1}\varepsilon_{2}\ldots$ is a $\beta$-expansion of $x$, or say that $x$ is the value of $0. \varepsilon_{1}\varepsilon_{2}\ldots$.

 To simplify the notation, instead of the equation (\ref{def}) we  write $$x=0. \varepsilon_{1}\varepsilon_{2}\ldots(\beta), $$ or for short  $$x=0. \varepsilon_{1}\varepsilon_{2}\ldots$$ if it causes no confusion. If an expansion ends by $0^{\infty}$ (i.e. there exists $n\in\mathbb{N}$ such that $\varepsilon_{k}=0$ for $k>n$), we call it a finite expansion; otherwise, we call it an infinite expansion. We always identify a finite expansion $0. \varepsilon_{1}\varepsilon_{2}\ldots\varepsilon_{n}00\ldots$ with   $0. \varepsilon_{1}\varepsilon_{2}\ldots\varepsilon_{n}$.

In this paper, the formula such as $0. \varepsilon_{1}\varepsilon_{2}\ldots=0. \eta_{1}\eta_{2}\ldots$ always means that the corresponding values are equal, while the formula such as  $\varepsilon_{1}\varepsilon_{2}\ldots=\eta_{1}\eta_{2}\ldots$ means the sequences are exactly the same, i.e. $\varepsilon_{1}=\eta_{1}, \varepsilon_{2}=\eta_{2}, \ldots$.

Given $x\in I_{\beta, m}$, the set of the expansions of $x$ is denoted by
 $$\mathcal{E}_{\beta, m}(x)=\{(\varepsilon_{i})_{i=1}^{\infty}\in\{0, 1, \ldots, m\}^{\infty}:
x=\sum_{i=1}^{\infty}\frac{\varepsilon_{i}}{\beta^{i}}\}.$$

 Let us remark that   we do not impose any other restrictions on the digits $\varepsilon_n$   than the equation (\ref{def}), and thus it was believed that any number will have more than one expansions. In   fact, it was proven that in the case that $\beta<m+1$, $a.e.~  x\in I_\beta$ has $2^{\aleph_{0}}$ such expansions, see \cite{DV, S}.

  Erd\H{o}s, Horv\'{a}th and Jo\'{o} \cite{EHJ} and Erd\H{o}s, Jo\'{o} and Komornik   \cite{EJK} showed that, with digits 0 or 1, when $1<\beta<\frac{\sqrt{5}+1}{2}$ each interior point of $I_{\beta}$ has a continuum of distinct expansions, and there exist infinitely many numbers $1<\beta<2$ for which the expansion of 1 is  unique. Any endpoint  of $[0, \frac{1}{\beta-1}]$ obviously has a unique expansion. Daroczy and Katai \cite{DK} showed that when $\beta\in(\frac{\sqrt{5}+1}{2}, 2]$ there exists $x\in(0, \frac{1}{\beta-1})$ such that the set $\mathcal{E}_{\beta, 1}(x)$ is singleton.
 Sidorov \cite{S2} proved that, for  $1<\beta\leq2$, the set of numbers $x\in I_{\beta}$ having less than a continuum of distinct expansions is
\begin{itemize}
   \item the two-point set of the endpoints of $I_{\beta}$ if $\beta< G$;
   \item countable infinite if $G\leq \beta<\beta_{c}$;
   \item a continuum of Hausdorff dimension 0 if $\beta=\beta_{c}$;
   \item a continuum of Hausdorff dimension strictly between 0 and 1 if $\beta_{c}<\beta<2$;
   \item the complementer of a countable set in $[0,1]$ if $\beta=2$,
 \end{itemize}
where $G=\frac{\sqrt{5}+1}{2}$ is the golden ratio, and $\beta_{c}\approx1.787\ldots$ is the ``Komornik-Loreti" constant, the smallest base under which the expansion of the number $1$  is unique  \cite{KL}. Moreover, $\beta_{c}$ is transcendental  \cite{AC}.
 Sidorov and  Vershik  \cite{SV} then studied the numbers with countable expansions under the base of golden ratio, and  showed that  for  $\beta=G=\frac{1+\sqrt{5}}{2}$ the set $\mathcal{E}_{\beta, 1}(x)$ is countable when $x=\beta n (\text{mod }1)$ for any $n\in\mathbb{Z}$, while for others $x$ in $(0, \frac{1+\sqrt{5}}{2})$ the set $\mathcal{E}_{\beta, 1}(x)$ is uncountable.

 Baker  \cite{B} generalized the results in  \cite{DK, EJK}. He defined a generalized  golden ratio $\mathcal{G}(m)$ for any  $m\in\mathbb{N}$,
$$\mathcal{G}(m)=\left\{
  \begin{array}{cl}
     k+1, & \hbox{ if  $m=2k$,} \\
    \frac{k+1+\sqrt{k^{2}+6k+5}}{2}, & \hbox{ if  $m=2k+1$,}
  \end{array}
\right.$$
and showed that for  $\beta\in(1, \mathcal{G}(m))$,  the set $\mathcal{E}_{\beta, m}(x)$ is uncountable  for each $x\in(0, \frac{m}{\beta-1})$, and for $\beta\in(\mathcal{G}(m), m+1]$ there exist $x\in(0, \frac{m}{\beta-1})$ such that the set $\mathcal{E}_{\beta, m}(x)$ is singleton.
Besides, in the same paper he posed the following  open problem:

\smallskip

Does an analogue of this  statement for numbers with exactly countably many expansions hold in the case of general $m\in \mathbb{N}? $

\smallskip

We consider this problem in this paper.    We have the  following main results.

\begin{theorem}\label{theorem2} Let  $m=2k$ be even,   $\beta=\mathcal{G}(m)=k+1$, and let $$\mathcal{F}=\{\frac{p}{(k+1)^{n}}\in(0, 2): n, p\in\mathbb{N}\cup\{0\}\}.$$
The elements of $\mathcal{F}$ have countably many expansions, while the other elements of $(0, 2)$ have uncountably many expansions.
\end{theorem}

\begin{theorem}\label{theorem} Let $m=2k+1$ be odd, $\beta=\mathcal{G}(m)=\frac{k+1+\sqrt{k^{2}+6k+5}}{2}$, and let $$\mathcal{S}=\{\frac{p\beta+q}{(k+1)^{n}}\in(0, \frac{m}{\beta-1}): n, p, q\in\mathbb{Z}\}.$$ The elements of $\mathcal{S}$ have countably many expansions, while the other elements of $(0,\beta-k)$ have uncountably many expansions.
\end{theorem}

Following the same idea as the proof of Theorem \ref{theorem}, we may prove Theorem \ref{theorem2}.
So, in this paper, we will devote ourselves to the proof of Theorem \ref{theorem}. In the next section,  some necessary  preliminaries are presented, and Theorems are proven in the last section.

\section{preliminaries}
We consider the case that $m=2k+1$ for $k\in\mathbb{N}$,  whence the digits set  $\Omega=\{0,1,\cdots, 2k+1\}$,
 the generalized golden ration $\beta=\mathcal{G}(m)=\frac{k+1+\sqrt{k^{2}+6k+5}}{2}$,  and $I_{\beta, m}=[0,\frac{m}{\beta-1}]=[0, \beta-k]$.
Recall that $\beta$ satisfies  the algebraic equation $$\frac{k+1}{\beta}+\frac{k+1}{\beta^{2}}=1.$$

There are several digit sets need to be considered. We define  the small-digit set   to be $S=\{0, 1, \ldots, k\}$,  and the big-digit set
 $B=\{k+1, k+2,  \ldots, 2k+1\}.$   Also we put  $S^{-}=\{0, 1, \ldots, k-1\} $ and $S_{-}=\{1,2,\ldots, k\}$ by removing the smallest element and the biggest one from $S$ respectively. In the same way,
 $B_-=\{k+2,\ldots,2k+1\}$ and $B^-=\{k+1,\ldots, 2k\}$.

\subsection{Sequence $F_{n}$}
In this subsection, we define a sequence $F_{n}$ which has some properties relating with $\beta$.
\begin{lemma}\label{n=fn}
Define  $\{F_{n}\}_{n\geq1}$ to be the integer sequence satisfying $$F_{n+1}=(k+1)(F_{n}+F_{n-1})\quad (n\geq2)$$ with $F_{1}=1, F_{2}=k+1$. Then for any integer $n\ge0$, there exists a finite sequence $\{n_{i}\}_{i=1}^{l}\in\{0, 1, \ldots, (k+1)\}^{\ast}$ such that
\begin{equation}\label{n}
n=\sum_{i=1}^{l}n_{i}F_{i},
\end{equation}
where $l=l(n)$ is dependent of  $n$.
\end{lemma}
We remark that the expansion (\ref{n}) of $n$ is by no means unique. While we can require that $l(n)<l$ if $n<F_{l}$.
\begin{proof}
The case    $n=0$  or $n=1$ is trivial. Now by induction, we assume that the conclusion holds for $0\leq n<F_{m}$.

Noticing the fact that $F_{m+1}<(k+2)F_{m}$ and $$[F_{m},(k+2)F_{m})=\bigcup_{j=0}^{k+1}[jF_{m}, (j+1)F_{m}).  $$
When $F_{m}\leq n <F_{m+1}$, we have that
$n\in[jF_{m}, (j+1)F_{m})$ for some $j\in\{0, 1, \ldots, k+1\}$, and $0\leq n-jF_{m}<F_{m}$. By the hypothesis of induction, $$n-jF_{m}=\sum_{i=1}^{m-1}\widetilde{n_{i}}F_{i}, $$ hence, we take $n_{m}=j$ together with $n_{i}=\widetilde{n_{i}}$ for $i<m$ to obtain an expansion of $n$.

\end{proof}

\begin{lemma}\label{fn}
For any $n\in\mathbb{N}$, $F_{n}\beta=F_{n+1}-(-\frac{k+1}{\beta})^{n}$.
\end{lemma}

\begin{proof}

Since $$F_{n+1}=(k+1)(F_{n}+F_{n-1})$$ and $$(-\frac{k+1}{\beta})^{n}=(k+1)(-\frac{k+1}{\beta})^{n-1}+(k+1)(-\frac{k+1}{\beta})^{n-2},  $$
the conclusion follows by induction.
\end{proof}
\subsection{Properties on the $\beta$-expansion}
\begin{lemma}\label{1}
The number $1$ has countably many expansions under the base $\beta$.
\end{lemma}
\begin{proof}
Recalling the fact  that $1.00=0.(k+1)(k+1), $ we have
$0. (k+2)=1. 00(k+1)$ and $0. (2k+1)(2k+1)\ldots=1. (k+1). $

Let $0. \delta_{1}\delta_{2}\ldots\in\{0, 1, \ldots, 2k+1\}^{\infty}$ be an expansion of $1$.
 We consider four cases according as the value of the the first digit $\delta_{1}$.
\begin{case}  $\delta_{1} \in S^{-}. $\end{case}

Since  $0. \delta_{2}\delta_{3}\ldots\leq0. (2k+1)(2k+1)(2k+1)\ldots=0.1(k+1)$,
\begin{eqnarray*}
  0. \delta_{1}\delta_{2}\ldots &\leq& 0. \delta_{1}(2k+1)(2k+1)(2k+1)\ldots \\
   &=& 0. (\delta_{1}+1)(k+1) \\
   &<& 0.(k+1)(k+1)=1.00.
\end{eqnarray*}
This case is impossible.

\begin{case}  $\delta_{1}=k.$\end{case} We know that $0. (2k+1)(2k+1)\ldots=1. (k+1)$, then  $0. \delta_{1}\delta_{2}\ldots=0. (k+1)(k+1)$. Thus in this case the only possibility  is that $0. \delta_{1}\delta_{2}\ldots=0. k(2k+1)(2k+1)\ldots. $

\begin{case}  $\delta_{1}=k+1.$\end{case}

\begin{itemize}
  \item $\delta_{1}=k+1, \delta_{2}\in S^{-}. $

  Since \begin{eqnarray*}
                                                                  0. (k+1)\delta_{2}(2k+1)(2k+1)\ldots &=& 0. (k+1)(\delta_{2}+1)(k+1) \\
                                                                    &<& 0.(k+1)(k+1).
                                                                 \end{eqnarray*}
  Thus this subcase is impossible.
  \item $\delta_{1}=k+1, \delta_{2}=k. $

  In this subcase, we readily check that  $0. \delta_{3}\delta_{4}\ldots$ is again an expansion of 1.
   \item $\delta_{1}=k+1, \delta_{2}=k+1. $

   Clearly, in this subcase the only possibility is that $1=0.(k+1)(k+1)000\ldots.$
  \item $\delta_{1}=k+1, \delta_{2}\in B_{-}.$

  Since $0. (k+1)\delta_{2}>0. (k+1)(k+1)=1, $ it is impossible.
\end{itemize}
\begin{case}
$\delta_{1}\in B_{-}. $
\end{case}
Since $0.(k+2)=1.00(k+1)>1, $ it is impossible.

In conclusion, the expansion of 1 takes one of the forms:
\begin{enumerate}
  \item $1=0. k(2k+1)(2k+1)(2k+1)\ldots$
  \item $1=0. (k+1)(k+1)$
  \item $1=0. (k+1)k \delta_{3}\delta_{4}\ldots$ with $0.\delta_{3}\delta_{4}\ldots $again an expansion of 1.
\end{enumerate}

By an easy induction(on the number of the block $((k+1)k)'$s occurring in the beginning of the expansion), we then obtain all of the expansion of 1 as follows
\begin{itemize}
  \item $0.((k+1)k)^{n}k(2k+1)(2k+1)(2k+1)\ldots$
  \item $0.((k+1)k)^{n}(k+1)(k+1)$
  \item $0.((k+1)k)^{\infty}, $ i.e. $0. (k+1)k (k+1)k (k+1)k \ldots$.
\end{itemize}

\end{proof}

In the following, we study the ``carry" and ``borrow" of the expansion in the light of the formula $1.00=0.(k+1)(k+1)$.

First we introduce a notation ``index" Ind$^{+}$($x$) for a  sequence $x=0.x_{1}x_{2}\ldots x_{n}\ldots\in\{0,1,\ldots, 2k+1\}^{\infty}$    as follows:  for $i\geq1$,
$$ \text{Ind}^{+}(x)=\left\{
  \begin{array}{ll}
    2i-1, & \hbox{if $x_{2j-1}\in S, x_{2j}\in B$ for $j=1,2,\ldots, i-1$ and $x_{2i-1}\in B$;} \\
    2i, & \hbox{if $x_{2j-1}\in S, x_{2j}\in B$ for $j=1,2,\ldots, i-1$ and $x_{2i-1}, x_{2i}\in S$;}\\
    \infty, & \hbox{if $x_{2j-1}\in S, x_{2j}\in B$ for $j=1,2,3,\ldots$.}\\
  \end{array}
\right.$$

Then for any sequence $0.bx_{1}x_{2}x_{3}\ldots$ with $b\in B_-=\{k+2,\ldots,2k+1\}$, we can define the
``carry" map $T^{+}$
 according as the value of the  Ind$^{+}$($x$) as follows:
\begin{itemize}

  \item If Ind$^{+}$($0. x_{1}x_{2}x_{3}\ldots$)=1, then
 $$T^{+}(0.bx_{1}x_{2}x_{3}\ldots) = 1.(b-k-1)(x_{1}-k-1)x_{2}x_{3}\ldots. $$

  \item If Ind$^{+}$($0. x_{1}x_{2}x_{3}\ldots$)=$2i-1$  for any  $i\geq2$, then
\begin{eqnarray*}
   &&  T^{+}(0.bx_{1}x_{2}x_{3}\ldots)\\
   &=& 1.(b-k-2)(x_{1}+1)(x_{2}-1)\ldots(x_{2i-1}+1) x_{2i-2}(x_{2i-1}-k-1)x_{2i}x_{2i+1}\ldots.
\end{eqnarray*}

  \item If Ind$^{+}$($0. x_{1}x_{2}x_{3}\ldots$)=$2i$  for any  $i\geq1$, then
\begin{eqnarray*}
   &&T^{+}(0.bx_{1}x_{2}x_{3}\ldots)  \\
   &=& 1.(b-k-2)(x_{1}+1)(x_{2}-1)\ldots x_{2i-1}(x_{2i}+k+1)x_{2i+1}x_{2i+2}\ldots.
\end{eqnarray*}

  \item If Ind$^{+}$($0. x_{1}x_{2}x_{3}\ldots$)=$\infty$, then
\begin{eqnarray*}
   &&T^{+}(0.bx_{1}x_{2}x_{3}\ldots)  \\
   &=& 1.(b-k-2)(x_{1}+1)(x_{2}-1)(x_{3}+1)(x_{4}-1)\ldots.
\end{eqnarray*}
\end{itemize}

Obviously, for any $k\in\{1,2,\cdots\}\cup\{\infty\}$ the restriction of the map $T^{+}$ to the sequences    $0.bx_{1}x_{2}x_{3}\ldots$ with $b\in B_-$ and Ind$^{+} ( 0. x_{1}x_{2}x_{3}\ldots)=k $ is injective.

In a dual way, we define the  Ind$^{-}$($x$): for $i\geq1$
$$ \text{Ind}^{-}(x)=\left\{
  \begin{array}{ll}
    2i-1, & \hbox{if $x_{2j-1}\in B, x_{2j}\in S$ for $j=1,2,\ldots, i-1$ and $x_{2i-1}\in S$;} \\
    2i, & \hbox{if $x_{2j-1}\in B, x_{2j}\in S$ for $j=1,2,\ldots, i-1$ and $x_{2i-1}, x_{2i}\in B$;}\\
    \infty, & \hbox{if $x_{2j-1}\in B, x_{2j}\in S$ for $j=1,2,\ldots$ .}
  \end{array}
\right.$$

Then for any sequence $1.a x_{1}x_{2}x_{3}\ldots$ with $a\in S^{-}=\{0,\ldots, k-1\}$, we   define the
``borrow" map $T^{+}$  as follows:

\begin{itemize}
  \item If Ind$^{-}$($0. x_{1}x_{2}x_{3}\ldots$)=1, then
 $$T^{-}(1.a x_{1}x_{2}x_{3}\ldots) = 0.(a+k+1)(x_{1}+k+1)x_{2}x_{3}\ldots. $$

  \item If Ind$^{-}$($0. x_{1}x_{2}x_{3}\ldots$)=$2i-1$  for any  $i\geq2$, then
\begin{eqnarray*}
   &&  T^{-}(1. a x_{1}x_{2}x_{3}\ldots)\\
   &=& 0.(a+k+2)(x_{1}-1)(x_{2}+1)\ldots (x_{2i-3}-1)x_{2i-2}(x_{2i-1}+k+1)x_{2i}\ldots.
\end{eqnarray*}

  \item If Ind$^{-}$($0. x_{1}x_{2}x_{3}\ldots$)=$2i$  for any  $i\geq1$, then
\begin{eqnarray*}
   &&T^{-}(1. a(k+2)x_{1}x_{2}x_{3}\ldots)  \\
   &=& 0. (a+k+2)(x_{1}-1)(x_{2}+1)\ldots (x_{2i-2}+1)x_{2i-1}(x_{2i}-k-1)x_{2i+1}\ldots.
\end{eqnarray*}

  \item  If Ind$^{-}$($0. x_{1}x_{2}x_{3}\ldots$)=$\infty$, then
\begin{eqnarray*}
   &&T^{-}(1. a(k+2)x_{1}x_{2}x_{3}\ldots)  \\
   &=& 0. (a+k+2)(x_{1}-1)(x_{2}+1)(x_{3}-1)(x_{4}+1)\ldots.
\end{eqnarray*}

\end{itemize}

\medskip

\begin{lemma}
The set $\{0, 1, \ldots, 2k+1\}^{\ast}$ is closed under  multiplication  by $\beta$, more precisely,  for any $x\in(0, \frac{\beta-k}{\beta})$ which has a finite expansion $x=0. \varepsilon_{1}\ldots\varepsilon_{n}$, there exist $\eta_{1}\ldots\eta_{m}\in\{0, 1, \ldots, 2k+1\}^{\ast}$ such that $\beta x =0. \eta_{1}\ldots\eta_{m}.$
\end{lemma}
\begin{proof}
Let  $x\in(0, \frac{\beta-k}{\beta})$ be a  number with a finite expansion, and   $0. \varepsilon_{1}\ldots\varepsilon_{n}$ be an expansion of $x$.  Then $0. \varepsilon_{1}\ldots\varepsilon_{n}<0. 1(k+1)$ since $\frac{\beta-k}{\beta}=0. 1(k+1). $

We need to find a finite expansion for  $\beta x$. We consider three cases according as the value of the first two  digits $\varepsilon_{1}\varepsilon_{2}.$
\addtocounter{case}{-4}
\begin{case} $\varepsilon_{1}=0$.
\end{case}
In this case, it is clearly that $\beta x=0.\varepsilon_{2}\ldots\varepsilon_{n}$.
\begin{case}\label{10}
$\varepsilon_{1}=1, \varepsilon_{2}\in S^{-}.$
\end{case}
 Recall that $0.(2k+1)(2k+1)(2k+1)\ldots=1.(k+1)$. In this case, the digits can take any value in $\{0, 1,2, \ldots, 2k+1\}$, since $$0.1\varepsilon_{2}(2k+1)(2k+1)\ldots(2k+1)<0.1(k+1).$$
 Putting Ind$^{-}$($\varepsilon_{3}\ldots\varepsilon_{n}$)=$s$,  we have that

\begin{itemize}
  \item If $s=1$, then $$\beta x=0.(\varepsilon_{2}+k+1)(\varepsilon_{3}+k+1)\varepsilon_{4}\varepsilon_{5}\ldots \varepsilon_{n};$$

  \item If $s=2i-1$  for some $i\geq2$,  then
\begin{eqnarray*}
   && \beta x \\
   &=& 0.(\varepsilon_{2}+k+2)(\varepsilon_{3}-1)(\varepsilon_{4}+1)\ldots (\varepsilon_{2i-1}-1)\varepsilon_{2i}(\varepsilon_{2i+1}+k+1)\varepsilon_{2i+2}\ldots \varepsilon_{n};
\end{eqnarray*}
  \item If $s=2i$ for some $i\geq 1$, then
\begin{eqnarray*}
   && \beta x \\
   &=& 0. (\varepsilon_{2}+k+2)(\varepsilon_{3}-1)(\varepsilon_{4}+1)\ldots (\varepsilon_{2i}+1)\varepsilon_{2i+1}(\varepsilon_{2i+2}-k-1)\varepsilon_{2i+3}\ldots \varepsilon_{n}.
\end{eqnarray*}

\end{itemize}

\begin{case}
$\varepsilon_{1}=1, \varepsilon_{2}=k. $
\end{case}
In this case, since $0.1k\varepsilon_{3}\varepsilon_{4}\ldots<0.1(k+1)$, we have the constraint that $$0.\varepsilon_{3}\varepsilon_{4}\ldots<1. $$

Since $1=0.(k+1)k(k+1)k(k+1)k\ldots, $ the expansion of $x$ is of one of the following forms:
\begin{enumerate}
  \item $x=0.1k((k+1)k)^{p}a\varepsilon_{2p+4}\ldots\varepsilon_{n}\text{ for }a\in S, $
  \item $x=0.1k((k+1)k)^{p}(k+1)b\varepsilon_{2p+5}\ldots\varepsilon_{n}$
  $\text{ for }b\in S^{-}, $

\end{enumerate}
where $p\in\mathbb{N}\cup\{0\}$. \

In the first subcase, we have $$\beta x=1. k((k+1)k)^{p}a\varepsilon_{2p+4}\ldots\varepsilon_{n}. $$
By $1.00=0.(k+1)(k+1)$, it follows that
 $$\beta x=0.(2k+1)\ldots(2k+1)(a+k+1)\varepsilon_{2p+4}\ldots\varepsilon_{n},$$
 which is in $\{0,1,\ldots,2k+1\}^{\ast}$.

 In the second subcase, we have
 \begin{eqnarray*}
  && \beta x \\
    &=& 1. k((k+1)k)^{p}(k+1)b \varepsilon_{2p+5}\ldots\varepsilon_{n} \\
    &=& 0.(2k+1)\ldots(2k+1)(2k+2)b \varepsilon_{2p+5}\ldots\varepsilon_{n}.
 \end{eqnarray*}
Writing $\varepsilon_{2p+5}\varepsilon_{2p+6}\ldots\varepsilon_{n}=x_{1}x_{2}\ldots x_{q}$ and Ind$^{-}$($x_{1}x_{2}\ldots x_{q}$)=$s$, we have that
\begin{itemize}
  \item If $s=1$, then $$\beta x=0.(2k+1)\ldots(2k+1)(2k+1)(b+k+1)(x_{1}+k+1)x_{2}\ldots x_{q};$$

  \item If $s=2i-1$  for some $i\geq2$,  then
\begin{eqnarray*}
   && \beta x \\
   &=& 0.(2k+1)\ldots(2k+1)(b+k+2)(x_{1}-1)(x_{2}+1)\ldots (x_{2i-3}-1)x_{2i-2}\\&&(x_{2i-1}+k+1)x_{2i}\ldots x_{q};
\end{eqnarray*}
  \item If $s=2i$ for some $i\geq 1$, then
\begin{eqnarray*}
   && \beta x \\
   &=& 0.(2k+1)\ldots(2k+1)(b+k+2)(x_{1}-1)(x_{2}+1)\ldots (x_{2i-2}+1)x_{2i-1}\\&&(x_{2i}-k-1)x_{2i+1}\ldots x_{q}.
\end{eqnarray*}

\end{itemize}

\end{proof}

Let $x=0. \varepsilon_{1}\ldots\varepsilon_{n}\in\{0, 1, \ldots, 2k+1\}^{\ast}$. If the condition that $\varepsilon_{i}\in S$ for $1\leq i\leq n-1 $ implies that $\varepsilon_{i+1}\in B$, which means the digit $k+1, k+2, \ldots, $ or $2k+1$ is  always separated by the digit 0, 1,$\ldots,$  or $k$, we say that  this sequence $0. \varepsilon_{1}\ldots\varepsilon_{n}$ is $B$-separated.

For   $x=0.\varepsilon_{1}\varepsilon_{2}\ldots\varepsilon_{n}$, we define
 $$l=\min\left\{1 \leq i\leq n-1: \varepsilon_{i}\varepsilon_{i+1}\in B^{2}\right\}$$ with  the convention that $\min\emptyset=\infty$.
 It is easy to see that $l=\infty$ when and only when the sequence $0.\varepsilon_{1}\varepsilon_{2}\ldots\varepsilon_{n}$ is $B$-separated.

  We then define an operator $C_{r}$ as follows:
\begin{eqnarray*}
   && C_{r}(0.\varepsilon_{1}\varepsilon_{2}\ldots\varepsilon_{n}) \\
   &=& \left\{
    \begin{array}{ll}
      0.\varepsilon_{1}\ldots\varepsilon_{l-2}(\varepsilon_{l-1}+1)(\varepsilon_{l}-(k+1))(\varepsilon_{l+1}-(k+1))\varepsilon_{l+2}\ldots\varepsilon_{n}, & \hbox{\text{ if }  $l < \infty$;} \\
      0.\varepsilon_{1}\varepsilon_{2}\ldots\varepsilon_{n}, & \hbox{\text{ if } $l=\infty$.}
    \end{array}
  \right.
\end{eqnarray*}
It is easy to see that the operator $C_{r}$ preserves the value. On the other hand when $l<\infty$,  the operator $C_{r}$  reduces by $2k+1$ the summation of all the digits in the expansion.  Thus there exists $k\in\mathbb{N}$ such that $$C_{r}^{k+1}(0.\varepsilon_{1}\varepsilon_{2}\ldots\varepsilon_{n})=C_{r}^{k}(0.\varepsilon_{1}\varepsilon_{2}\ldots\varepsilon_{n}).$$
Whence the sequence $C_{r}^{k}(0.\varepsilon_{1}\varepsilon_{2}\ldots\varepsilon_{n})$ is $B$-separated,  and then we define the map $T$ as
$$T(0.\varepsilon_{1}\varepsilon_{2}\ldots\varepsilon_{n})=C_{r}^{k}(0.\varepsilon_{1}\varepsilon_{2}\ldots\varepsilon_{n}), $$ for such $k$.

If $\varepsilon_{1}\varepsilon_{2}\in  B^{2}$, then $0.\varepsilon_{1}\varepsilon_{2}=1.(\varepsilon_{1}-k-1)(\varepsilon_{2}-k-1)$. Denote $T(0.\varepsilon_{1}\varepsilon_{2}\ldots\varepsilon_{n})=\eta_{0}.\eta_{1}\ldots\eta_{n}$, where  $\eta_{0}\in\{0, 1\}$ and $\eta_{i}\in\{0,\ldots,2k+1\}$ for $1\leq i\leq n$. Clearly,  the map $T$ satisfies the following properties:
\begin{itemize}
\item If  $\eta_{i}\in B_{-}$ for some $i$, then $\varepsilon_{i}=\eta_{i}$. In other words, there is no new occurrence of the digit $\varepsilon\in B_{-}$ in the process, and thus we have $$|\eta_{0}.\eta_{1}\ldots\eta_{n}|_{\varepsilon}\leq|0.\varepsilon_{1}\varepsilon_{2}\ldots\varepsilon_{n}|_{\varepsilon}, $$ where $|\cdot|_{\varepsilon}$ denotes the total number of occurrences of the digit $\varepsilon$ in the sequence.
\item The sequence $\eta_{0}.\eta_{1}\ldots\eta_{n}$, or equivalently, the sequence $0.\eta_{1}\ldots\eta_{n}$ is $B$-separated.
\end{itemize}
\begin{lemma}\label{finite}
For any $0. \varepsilon_{1}\ldots\varepsilon_{n}\in\{0, 1, \ldots, 2k+1\}^{\ast}$, there exists $\widetilde{\varepsilon_{0}}. \widetilde{\varepsilon_{1}}\ldots\widetilde{\varepsilon_{m}}$ such that $\widetilde{\varepsilon_{0}}. \widetilde{\varepsilon_{1}}\ldots\widetilde{\varepsilon_{m}}=0. \varepsilon_{1}\ldots\varepsilon_{n}$,  where $\widetilde{\varepsilon_{0}}\in\{0, 1\}$ and $\widetilde{\varepsilon_{i}}\in\{0, 1, \ldots, k+1\}$ for $i=1, \ldots, m$.
\end{lemma}

\begin{proof}
Applying the map $T$ on the expansion $0. \varepsilon_{1}\ldots\varepsilon_{n}$ if necessary, we can suppose without loss of generality that $0. \varepsilon_{1}\ldots\varepsilon_{n}$ is $B$-separated. Our aim is to eliminate all the digits $k+2, \ldots, 2k+1$ from the expansion.

Now we want to eliminate the digit $2k+1$ in the first step.  We regroup the expansion $0. \varepsilon_{1}\ldots\varepsilon_{n}$ according as the position of the last $2k+1$ as follows:
\begin{equation}\label{eta}
0. \varepsilon_{1}\ldots\varepsilon_{n}=0. \varepsilon_{1}\ldots\varepsilon_{m}(2k+1) \delta_{1}\ldots\delta_{p},
\end{equation}
where  $m+k+1=n$ and  $\delta=\delta_{1}\ldots\delta_{p}\in\{0, 1, \ldots, 2k\}^{\ast}$. Moreover, we have that   $\varepsilon_{m}, \delta_{1}\in S$ from  the $B$-separation property.

We claim that we  can eliminate the last  occurrence of the digit $2k+1$ from the expansion without change of the value. Meanwhile,  in the process  there are no new occurrences of $2k+1$ in the resulted expansion. We will show the claim by induction on the length of $\delta$.

When $ \delta=\emptyset$, that is, $m+1=n$, recalling the fact  that $0.(2k+1)=1.(k-1)0(k+1)$, we have
$$0. \varepsilon_{1}\ldots\varepsilon_{n}=0. \varepsilon_{1}\ldots\varepsilon_{n-2}(\varepsilon_{n-1}+1)(k-1)0(k+1). $$
When the length of $\delta$ is 1,  we have
$$0. \varepsilon_{1}\ldots\varepsilon_{n}=0. \varepsilon_{1}\ldots\varepsilon_{m-1}(\varepsilon_{m}+1)(k-1)\delta_{1}(k+2). $$
Now by induction, we assume that the conclusion holds for the expansion with   the length of $\delta$  less than $p$.
When the length of $\delta$ is  $p$, 
according to the value of the digit $\delta_{2}$,  we consider the following three cases:

\addtocounter{case}{-3}
\begin{case}$\delta_{2}\in S^{-}$. \end{case}
In this case, by $0.(k+2)=1.00(k+1)$, we have
$$0. \varepsilon_{1}\ldots\varepsilon_{n}=0. \varepsilon_{1}\ldots\varepsilon_{m-1}(\varepsilon_{m}+1)(k-1)\delta_{1}(\delta_{2}+k+1)\delta_{3}\ldots\delta_{p}. $$
Thus the conclusion follows.
\begin{case}$\delta_{2}=k. $\end{case}
By $0.(k+2)=1.00(k+1)$, we have
$$0. \varepsilon_{1}\ldots\varepsilon_{n}=0. \varepsilon_{1}\ldots\varepsilon_{m-1}(\varepsilon_{m}+1)(k-1)\delta_{1}(2k+1)\delta_{3}\ldots\delta_{p}. $$
Applying  the map $T$ on the above expansion,  we have the conclusion by the hypothesis of induction.

\begin{case}$\delta_{2}\in B^{-}. $\end{case}
Recalling  the fact that the sequence $0. \varepsilon_{1}\ldots\varepsilon_{n}$  is $B$-separated, the block $\delta$ can be regrouped   as
\begin{equation}
\delta_{1}a_{1}\delta_{3}a_{2}\delta_{5}a_{2}\delta_{7}\ldots a_{t}\delta_{2t+1}\delta_{2t+2} \ldots \delta_{p},
\end{equation}
where $a_{j}\in B^{-}$ for $j\in\{1, 2, \ldots, t\}$,  and when $p \geq 2t+2$, $\delta_{2t+2}\in S$.

By $0.(k+2)=1.00(k+1)$, we obtain that
\begin{eqnarray*}
   && 0.(2k+1)\delta \\
   &=& 1.(k-1)(\delta_{1}+1)(a_{1}-1)\ldots (a_{t}-1)\delta_{2t+1}(\delta_{2t+2}+k+1)\delta_{2t+3} \ldots \delta_{p}.
\end{eqnarray*}

All the digits in the latter expansion are in $\{0, 1, \ldots, 2k\}$ expect at most that $\delta_{2t+2}+k+1\in\{0,1,\ldots,2k+1\}$. If $\delta_{2t+2}+k+1<2k+1$, we are done; otherwise  we apply the map $T$, and then  use the induction hypothesis.

 The claim then follows.

Up to now, we have already eliminate the last digit $2k+1$ in the expansion without new occurrence of $2k+1$ , and then we continue this process to eliminate all the digits $2k+1$ from the expansion.

Using the same argument, we eliminate the  other digit in $\{k+2, \ldots, 2k\}$.

\end{proof}

From the compactness of the symbol space $\{0, 1, \ldots, k+1\}^{\infty}$ ,we remark that any infinite sequence in $\{0, 1, \ldots, 2k+1\}^{\infty}$ has the property as in Lemma \ref{finite}, too.

\begin{remark}\label{infinite}
For any $0. \varepsilon_{1}\varepsilon_{2}\ldots\in\{0, 1, \ldots, 2k+1\}^{\infty}$, there exists $\widetilde{\varepsilon_{0}}. \widetilde{\varepsilon_{1}}\widetilde{\varepsilon_{2}}\ldots$ such that
$$0. \varepsilon_{1}\varepsilon_{2}\ldots=\widetilde{\varepsilon_{0}}. \widetilde{\varepsilon_{1}}\widetilde{\varepsilon_{2}}\ldots, $$
where $\widetilde{\varepsilon_{i}}\in\{0, 1, \ldots, k+1\}$ for $i\geq1$ and $\widetilde{\varepsilon_{0}}\in\{0,1\}$.
\end{remark}

\begin{lemma}\label{addition}
The set $\{0, 1, \ldots, 2k+1\}^{\ast}$ is closed under addition, more precisely, for any two finite expansions $\xi, \eta$ in $\{0, 1, \ldots, 2k+1\}^{\ast},$
there exists a sequence $\delta_{0}. \delta_{1}\delta_{2}\ldots\delta_{s}$ such that  $$\xi+\eta=\delta_{0}. \delta_{1}\delta_{2}\ldots\delta_{s},$$ where $0. \delta_{1}\delta_{2}\ldots\delta_{s}$ is in $\{0, 1, \ldots, 2k+1\}^{\ast}. $
\end{lemma}
\begin{proof}
Adding    0's at the end of  the expansion if necessary, we cam assume  without loss of generality that  the length of  $\xi$ and $\eta$ are  equal.

By Lemma \ref{finite}, we can suppose without loss of generality that
$x=\xi_{0}. \xi_{1}\ldots\xi_{n}$, $y=\eta_{0}. \eta_{1}\ldots\eta_{n}$, where $0.\xi_{1}\ldots\xi_{n}$ and $0. \eta_{1}\ldots\eta_{n}$ are in $\{0, 1, \ldots, k+1\}^{\ast}$.
Then
 $$\xi+\eta=z_{0}. z_{1}z_{2}\ldots z_{n}, $$
where
$z_{j}=\xi_{j}+\eta_{j}\in\{0, 1, 2, \ldots, 2k+2\}$ for $j=0,1,2,3,\ldots, n$.

Now we define two new  sequences as follows.
For any $j\in\{1,2, \ldots, n\}$, define

$$x_{j}=\left\{
  \begin{array}{ll}
    2k+1, & \hbox{ if $z_{j}=2k+2$;} \\
    z_{j}, & \hbox{ \text{others} .}
  \end{array}
\right. \text{and }
y_{j}=\left\{
  \begin{array}{ll}
    1, & \hbox{ if $z_{j}=2k+2$;} \\
    0, & \hbox{ \text{others} .}
  \end{array}
\right.$$

Obviously, we have $z_{j}=x_{j}+y_{j}$ for $j=1,2,3,\ldots, n$
and $0.y_{1}\ldots y_{n}\in\{0, 1\}^{\ast}$,   $0.x_{1}\ldots x_{n}\in\{0, 1, \ldots, 2k+1\}^{\ast}.$

By the   Lemma $\ref{finite}$ again, there is a sequence $\widetilde{x_{0}}. \widetilde{x_{1}}\ldots\widetilde{x_{q}}$ such  that
$$0.x_{1}\ldots x_{n}=\widetilde{x_{0}}. \widetilde{x_{1}}\ldots\widetilde{x_{q}},$$
where $0. \widetilde{x_{1}}\ldots\widetilde{x_{q}}\in\{0, 1, \ldots, k+1\}^{\ast}$. Hence, it follows   $$\xi+\eta=(\xi_{0}+\eta_{0}+\widetilde{x_{0}}).(\widetilde{x_{1}}+y_{1})\ldots(\widetilde{x_{q}}+y_{q})\quad (q>n),$$
where  $0.(\widetilde{x_{1}}+y_{1})\ldots(\widetilde{x_{q}}+y_{q})\in\{0, 1, \ldots, k+2\}^{\ast}.$
\end{proof}

\begin{lemma}$\label{dividingk}$
The set $\{0, 1, \ldots, 2k+1\}^{\ast}$ is closed under dividing by  $k+1$.
\end{lemma}
\begin{proof}
For each $\varepsilon\in\{0, 1, \ldots, 2k+1\}$, we define

$$i(\varepsilon)=\left\{
  \begin{array}{ll}
    0, & \hbox{if $\varepsilon\in S$;} \\
    k+1, & \hbox{if $\varepsilon\in B$.}
  \end{array}
\right.
\text{ and }
t(\varepsilon)=\left\{
  \begin{array}{ll}
    (k+1)\varepsilon, & \hbox{if $\varepsilon\in S$;} \\
    (k+1)(\varepsilon-k-1), & \hbox{if $\varepsilon\in B$.}
  \end{array}
\right.
$$

Let  $x=0. \varepsilon_{1}\ldots\varepsilon_{n}\in\{0, 1, \ldots, 2k+1\}^{\ast}$. Define $0.\eta_{1}\ldots\eta_{n+2}$ as
$$\eta_{j}=i(\varepsilon_{j})+t(\varepsilon_{j-1})+t(\varepsilon_{j-2})\quad(j=1, 2,\ldots, n+2),$$
with  $\varepsilon_{-1}=\varepsilon_{0}=0. $ Clearly, the sequence $0.\eta_{1}\ldots\eta_{n+2}$ has the following two properties :
\begin{itemize}
  \item  $\eta_{j}\in\{0, k+1, 2(k+1), \ldots, (2k+1)(k+1)\}$,
  \item $0.\eta_{1}\ldots\eta_{n+2}= 0. \varepsilon_{1}\ldots\varepsilon_{n}$.
\end{itemize}

Thus, we obtain
 $$\frac{x}{k+1}=0.\frac{\eta_{1}}{k+1}\ldots\frac{\eta_{n+2}}{k+1}\in\{0, 1, \ldots, 2k+1\}^{\ast}.$$
\end{proof}
\begin{corollary}\label{2n}
For any $n\in\mathbb{N}\cup\{0\}$, $\frac{1}{(k+1)^{n}}$  has at least a finite expansion.
\end{corollary}
\begin{proof}
This is a direct result of Lemma \ref{1} and   $\ref{dividingk}$.

\end{proof}
\section {proof of theorem \ref{theorem}}

We are now in a position to proof Theorem \ref{theorem}. Recalling that
$$\mathcal{S}=\{\frac{p\beta+q}{(k+1)^{n}}\in(0, \beta-k): n, p, q\in\mathbb{Z}\}.$$

\begin{lemma}\label{xs}
Let  $x \in (0, \beta-k)$. Then $x$ has a finite expansion if and only if $x\in\mathcal{S}$.
\end{lemma}
\begin{proof}
Let $x\in(0, \beta-k)$ be a number with a finite expansion  $0. \varepsilon_{1}\ldots\varepsilon_{2n}$, i.e.,
$$x=\sum_{i=1}^{2n}\varepsilon_{i}\beta^{-i}, $$ where $\varepsilon_{i}\in\{0, 1, \ldots, 2k+1\}$ for $i=1, 2, \ldots, 2n $. By Lemma $\ref{fn}$, we have

\begin{eqnarray*}
  x &=& \sum_{i=1}^{n}\frac{\varepsilon_{2i-1}(F_{2i-1}\beta-F_{2i})}{(k+1)^{2i-1}}+\sum_{i=1}^{n}\frac{\varepsilon_{2i}(F_{2i+1}-F_{2i}\beta)}{(k+1)^{2i}} \\
   &=& \frac{1}{(k+1)^{2n}}[(\sum_{i=1}^{n}\varepsilon_{2i-1}F_{2i-1}(k+1)^{2n-(2i-1)}-\sum_{i=1}^{n}\varepsilon_{2i}F_{2i}(k+1)^{2n-2i})\beta \\ &+&(\sum_{i=1}^{n}\varepsilon_{2i}F_{2i+1}(k+1)^{2n-2i}-\sum_{i=1}^{n}\varepsilon_{2i-1}F_{2i}(k+1)^{2n-(2i-1)})]
\end{eqnarray*}

Putting  $$p=\sum_{i=1}^{n}\varepsilon_{2i-1}F_{2i-1}(k+1)^{2n-(2i-1)}-\sum_{i=1}^{n}\varepsilon_{2i}F_{2i}(k+1)^{2n-2i}$$ and $$q=\sum_{i=1}^{n}\varepsilon_{2i}F_{2i+1}(k+1)^{2n-2i}-\sum_{i=1}^{n}\varepsilon_{2i-1}F_{2i}(k+1)^{2n-(2i-1)}, $$ we show that, $x\in\mathcal{S}$.

On the other hand,  if $x\in\mathcal{S}$, then there exist $p, q, n\in\mathbb{Z}$ such that  $$x=\frac{p\beta+q}{(k+1)^{n}} $$
with  $x\in(0, \beta-1)$. Applying   Lemma \ref{n=fn} with  $p$, we have
\begin{equation}\label{p}
   p=\sum_{i=1}^{2p_{0}}n_{i}F_{i},
\end{equation}
where $n_{i}\in\{0, 1, \ldots, k+1\}$ for any $i\in\{1, \ldots, 2p_{0}\}$,  or  $n_{i}\in\{0, -1, \ldots, -(k+1)\}$ for any $i\in\{1, \ldots, 2p_{0}\}$.
From Lemma $\ref{fn}$ and (\ref{p}), we have
\begin{eqnarray*}
  x &=&\frac{1}{(k+1)^{n}}\left[\sum_{i=1}^{p_{0}}n_{2i-1}(F_{2i}+(\frac{k+1}{\beta})^{2i-1})+\sum_{i=1}^{p_{0}}n_{2i}(F_{2i+1}-(\frac{k+1}{\beta})^{2i})+q\right] \\
   &=& \sum_{i=1}^{p_{0}}\frac{n_{2i-1}}{\beta^{2i-1}(k+1)^{n-(2i-1)}}- \sum_{i=1}^{p_{0}}\frac{n_{2i}}{\beta^{2i}(k+1)^{n-2i}}+\frac{M}{(k+1)^{n}},
\end{eqnarray*}
where $M=\sum_{i=1}^{2p_{0}}n_{i}F_{i+1}+q$.
Hence, from Lemma \ref{addition} and Corollary \ref{2n}, we obtain  that  any $x\in\mathcal{S}$ has a finite expansion.
\end{proof}
\begin{remark}\label{finite-coun}
By Lemma \ref{1}, the number 1 has countably many expansions, and thus for any $x\in S$, we can deduce at least countably many expansions for $x$ from its finite expansion.
\end{remark}

\begin{lemma}\label{coun-finite}
If  $x\in(0, \beta-k)$ has countably many  expansions then $x$ at least has a finite expansion.
\end{lemma}

\begin{proof}

Let
\begin{equation}\label{x}
    0. \varepsilon_{1}\varepsilon_{2}\varepsilon_{3}\ldots
\end{equation}
be an infinite expansion of $x$. We consider the 2-blocks appearing in the expansion.
\addtocounter{case}{-3}

\begin{case}\label{bb}
There are infinite many blocks in $B^2$ or  $S^2$ appearing in (\ref{x}).
\end{case}
\begin{enumerate}
  \item There are  infinite many  blocks of form  $xb_{1}b_{2}$ with $x\in\{0,1, \ldots, 2k\}$, $b_{1}, b_{2}\in B$, or $ya_{1}a_{2}$ with  $y\in\{1,2, \ldots, 2k+1\}$, $a_{1},a_2\in S. $

Notice that $$0.xb_{1}b_{2}=0.(x+1)(b_{1}-k-1)(b_{2}-k-1),  $$ and $$0.ya_{1}a_{2}=0.(y-1)(a_{1}+k+1)(a_{2}+k+1). $$ The number  $x$ has uncountably many  expansions.
  \item The expansion  ends with $(2k+1)^{\infty}$ or $0^{\infty}$.

Since $1=0.k(2k+1)^{\infty}$, the number  $x$ has a finite expansion.
\end{enumerate}

\begin{case}
The blocks in $B^2\cup S^2$ appears for finite times.
\end{case}
In this case, the expansion is of the type
\begin{equation}\label{x1}
x=0.\varepsilon_{1}\ldots\varepsilon_{n_{1}}a^{(1)}_{1}b^{(1)}_{1}a^{(1)}_{2}b^{(1)}_{2}a^{(1)}_{3}b^{(1)}_{3}\ldots,
\end{equation}
with  $a^{(1)}_{i}\in S$ and $b^{(1)}_{i}\in B$ for $i\geq 1$.

Put $l=\min\{i\geq1:  b_{i}\geq k+2 \}$  with the convention that $\min\emptyset=\infty. $ Thus  $l=\infty$ when and only when  $b_{i}=k+1$ for all $i\geq1$.

When $l=\infty$, we have
\begin{eqnarray*}
  x &=&  0.\varepsilon_{1}\ldots\varepsilon_{n_{1}}a^{(1)}_{1}(k+1)a^{(1)}_{2}(k+1)a^{(1)}_{3}(k+1)\ldots\\
   &=& 0.\varepsilon_{1}\ldots\varepsilon_{n_{1}}a^{(1)}_{1}k (a^{(1)}_{2}+k+1)(2k+1)(a^{(1)}_{3}+k+1)(2k+1)\ldots
\end{eqnarray*}

There are infinitely many blocks in $B^2$, whence just as in Case \ref {bb}, either $x$  has a finite expansion, or
 $x$ has uncountably many  expansions.

When $l<\infty$, we have,  by $0.(k+2)=1.00(k+1)$, that
\begin{eqnarray*}
  \mathcal{E}_{1} &=& 0.\varepsilon_{1}\ldots\varepsilon_{n_{1}}a^{(1)}_{1}b^{(1)}_{1}\ldots a^{(1)}_{l-1}b^{(1)}_{l-1}(a^{(1)}_{l}+1)0(a^{(1)}_{l+1}+1)(b^{(1)}_{l+1}-1)\ldots.
\end{eqnarray*}
We rewrite this expansion as
\begin{equation}\label{x2}
    \mathcal{E}_{2}= 0.\varepsilon_{1}\ldots\varepsilon_{n_{2}}a^{(2)}_{1}0a^{(2)}_{2}b^{(2)}_{2}a^{(2)}_{3}b^{(2)}_{3}\ldots,
\end{equation}
with  $$a^{(2)}_{i}\in S+1=\{x+1: x\in S\}=\{1, 2, \ldots, k+1\}$$  and  $$b^{(2)}_{i}\in B-1=\{y-1: y\in B\}=\{k, k+1, \ldots, 2k
\}$$.

Meanwhile, using the same argument as the Case\ref{bb}, we may assume that the blocks in $B^2\cup S^2$ appears for only finite times, and thus in the expansion (\ref{x2}), the subword $a^{(2)}_{i}b^{(2)}_{i}$ for $i$ large  enough
\begin{itemize}
  \item  can not in $\{1k, 2k, \ldots, kk, (k+1)(k+1), (k+1)(k+2), \ldots, (k+1)(2k+1)\}$,
  \item may not be equal to $(k+1)k$, expect  the case  when   $a^{(2)}_{i}b^{(2)}_{i}=(k+1)k$ eventually. Whence by $1.00=0.((k+1)k)^{\infty}$, the expansion can be transformed into a finite one.
\end{itemize}

In the light of these, we may suppose that the expansion of $x$ is of  the form:
\begin{equation}\label{x3}
    x= 0.\varepsilon_{1}\ldots\varepsilon_{n_{3}}a^{(3)}_{1}b^{(3)}_{1}a^{(3)}_{2}b^{(3)}_{2}a^{(3)}_{3}b^{(3)}_{3}\ldots,
\end{equation}
with $$a^{(3)}_{i}\in (S+1)^{-}=\{1, 2, \ldots, k\}$$  and  $$b^{(3)}_{i}\in (B-1)_{-}=\{k+1, \ldots, 2k
\}$$ for $i\geq 1$.

Up to now, we only need to consider the expansion (\ref{x3}). Compared with expansion (\ref{x1}), all $a_{i}b_{i}$ may take the values in $(S+1)^-\times(B-1)_{-}$ rather than $SB$. We continue this process to reach an expansion with all $a_{i}b_{i}\in (S+2)^-\times(B-2)_{-}$, and so  on. Finally, we obtain an expansion ending by $((k+1)k)^{\infty}, $ and it can be transformed into a finite one as before.

\end{proof}

\begin{lemma}\label{finite-coun}
Any number in $\mathcal{S}$  has countably many  expansions.
\end{lemma}
\begin{proof} By Remark \ref{finite-coun}, we only need show that any number in $\mathcal{S}$ can not have uncountably many expansions.

To this end, we suppose that
  $x\in\mathcal{S} $ has have uncountably many expansions. Considering for $\beta^{-n}x$ instead, we may assume $x<1.$ Thus
  $$y=1-x\in\mathcal{S}.$$
  By Lemmas \ref{xs} and \ref{finite}, the number $y$ has a finite expansion
   $$y=0. \varepsilon_{1}\ldots\varepsilon_{n}\in\{0,1,\ldots,k+1\}^{\ast}.$$

Since $x$ has uncountable many  expansions, there exist an uncountable index set $\Lambda$ and a block $\omega=\omega_{1}\omega_{2}\ldots\omega_{n}\in\{0,1, \ldots, 2k+1\}^{\ast}$ such that for any $\lambda\in\Lambda$,
 $$\omega_{\lambda}=0. \omega_{1}\omega_{2}\ldots\omega_{n}\omega_{n+1}^{\lambda}\omega_{n+2}^{\lambda}\ldots$$
 is an expansion of $x$, and thus
$$1=x+y=0.z_{1}z_{2}\ldots z_{n}\omega_{n+1}^{\lambda}\omega_{n+2}^{\lambda}\ldots, $$
where $z_{j}=\omega_{j}+\varepsilon_{j}\in\{0,1,\ldots,3k+2\}$ for $1\leq j\leq n$.

\smallskip

Applying the map $T$ to $0.z_{1}z_{2}\ldots z_{n}$, we obtain that $$T(0.z_{1}z_{2}\ldots z_{n})=\widetilde{z_{0}}. \widetilde{z_{1}}\widetilde{z_{2}}\ldots\widetilde{z_{n}}$$
and if $\widetilde{z_{i}}>k+1$, then we have
$$\left\{
   \begin{array}{ll}
     \widetilde{z_{i+1}}\leq k, & \hbox{for $i=1,2,\ldots, n$;} \\
     \widetilde{z_{i-1}}\leq k, & \hbox{for $i=2,3,\ldots, n$.}
   \end{array}
\right.\eqno{(*)}$$
Since $0.z_{1}z_{2}\ldots z_{n}<x+y=1$ and $T$ preserves the value, $\widetilde{z_{0}}=0$. So we may
suppose that the expansion $0.z_{1}z_{2}\ldots z_{n}$ satisfies the above property $(*)$.

\smallskip

We claim that the number $f_{n-1}=1-0.z_{1}z_{2}\ldots z_{n-1}$ has uncountably many expansions.

Since Ind$^{+}(\cdot)$ takes values amongst a countable set, there exists
an uncountable subset $\Lambda_{1}$ of $\Lambda$ such that
Ind$^{+}(\omega_{n+1}^{\lambda}\omega_{n+2}^{\lambda}\ldots)=s$ for some $s\in\{1,2,\ldots\}\cup\{\infty\}$.

If $z_{n}\leq2k+1$, then $$0.0\ldots0z_{n}\omega_{n+1}^{\lambda}\omega_{n+2}^{\lambda}\ldots\in\{0,1,\ldots, 2k+1\}^{\infty}$$ for $\lambda\in\Lambda_{1}$. Hence, $f_{n-1}$ has uncountably many expansions.

If $z_{n}>2k+1$, then for $\lambda\in\Lambda_{1}$, 
$$T^+(0.0\ldots0z_{n}\omega_{n+1}^{\lambda}\omega_{n+2}^{\lambda}\ldots)=0.0\ldots01\xi_{n}\xi_{n+1}\ldots,$$
with $\xi_{i}\in\{0,1,\ldots,2k+1\}$.  Since the restriction of the map $T^{+}$ to sequences with Ind$^{+}(\cdot)=s$ is an injection, these provide   uncountably many expansions of $f_n$.

 The claim then follows. Now we show each $f_{l}=1-0.z_{1}z_{2}\ldots z_{l}$ for $l\in\{ n-1,\ldots,1,0\}$ has uncountable many expansions. For this, we consider the following property  $(P_q)$:

 there exists an uncountable set
 $\Lambda_{q}$ such that
 \begin{itemize}
   \item if   $z_{q+1}\leq 2k+1$, then $f_q$ has the expansions of the form $$ 0.0\ldots00\omega_{q+1}^{\lambda}\omega_{q+2}^{\lambda}\ldots ~(\lambda\in\Lambda_{q});$$
   \item  $z_{q+1}>2k+1$,   then $f_q$ has the expansions of the form $$ 0.0\ldots01\omega_{q+1}^{\lambda}\omega_{q+2}^{\lambda}\ldots ~( \lambda\in\Lambda_{q}).$$
 \end{itemize}

\smallskip

 Suppose   $(P_q)$ holds, then
\begin{eqnarray*}
  f_{q-1} &=& 1-0.z_{1}z_{2}\ldots z_{q-1} \\
&=& 0.0\ldots0(z_{q}+a)\omega_{q+1}^{\lambda}\omega_{q+2}^{\lambda}\ldots
\end{eqnarray*}
where $a=0$  if  $z_{q+1}\leq 2k+1$, and   $a=1$  if  $z_{q+1}> 2k+1$.

If  $z_{q+1}\leq 2k+1$, we reach the property $(P_{q-1})$ using the same argument as above;
if $z_{q+1}> 2k+1$,  then due to Property $(*)$, we have $z_{q}\leq k$. Whence $0.0\ldots0(z_{q}+a)\omega_{q+1}^{\lambda}\omega_{q+2}^{\lambda}\ldots $ is the desired expansion, and we obtain $(P_{q-1})$ also.

Therefore, by induction, we know that any $f_{l}=1-0.z_{1}z_{2}\ldots z_{l}$  has uncountable many expansions. This is a contradiction since $f_0=1$ has only countably many expansions.
\end{proof}

Theorem \ref{theorem} then follows from the above lemmas.

\medskip

Following the similar idea with the proof of Theorem \ref{theorem} and with even less effort, we prove Theorem \ref{theorem2}.

\medskip

\begin{flushleft}
{\bf Acknowledgements}~ This work was supported by  NSFC Nos. 11171123 and 11222111.
\end{flushleft}

\end{document}